\documentclass[reqno]{amsart}

\usepackage{amsmath,amssymb,amsfonts,amsthm}
\usepackage[dvipsnames]{xcolor}
\usepackage{tikz}
\usetikzlibrary{arrows.meta}

 \usepackage{enumitem}

\DeclareFontFamily{OT1}{pzc}{}
\DeclareFontShape{OT1}{pzc}{m}{it}{<-> s * [0.900] pzcmi7t}{}
\DeclareMathAlphabet{\mathscr}{OT1}{pzc}{m}{it}

\usepackage{versions}

\usepackage{graphicx}
\usepackage{grffile}

\includeversion{note}
\includeversion{longcalc}
\excludeversion{shortcalc}

\usepackage[utf8]{inputenc}

\usepackage[nonotes]{optional}

\newcommand{\arjun}[1]{\opt{notes}{\todo[size=\tiny,backgroundcolor=white,noline]{\textcolor{Mahogany}{arjun: #1}}}}
\newcommand{\notes}[1]{\opt{notes}{\textcolor{OliveGreen}{#1}}}

\usepackage[sorting=none,natbib=true,style=numeric,backend=biber]{biblatex}

\makeatletter
\@namedef{subjclassname@2020}{%
  \textup{2020} Mathematics Subject Classification}
\makeatother

\usepackage[normalem]{ulem} %

\newcommand{\Figref}[1]{Fig.~\ref{#1}}

\newcommand{\Thmref}[1]{Theorem~\ref{#1}}

\newcommand{\Lemref}[1]{Lemma~\ref{#1}}

\newcommand{\R}{\mathbb{R}} %
\newcommand{\Z}{\mathbb{Z}} %
\newcommand{\Prob}{\mathbb{P}} %

\newcommand{\w}{\omega}

\newcommand{\AND}{\textrm{ and }} 
\newcommand{\OR}{\textrm{ or }}

\renewcommand{\liminf}{\varliminf}

\newcommand{\E}{\mathbb{E}}

\renewcommand{\emph}[1]{\textbf{#1}}

\newcommand{\arrow}{\alpha}

\theoremstyle{plain}
\newtheorem{theorem}{Theorem}[section]
\newtheorem*{theorem*}{Theorem}
\newtheorem{lemma}[theorem]{Lemma}
\newtheorem*{lemma*}{Lemma}

\newtheorem*{prop*}{Proposition}

\theoremstyle{definition}

\newtheorem*{question*}{Question}

\newtheorem*{example*}{Example}
\newtheorem{claim}[theorem]{Claim}

\theoremstyle{remark}
\newtheorem{remark}{Remark}
\newtheorem*{remark*}{Remark}
\newtheorem*{acknowledgement}{Acknowledgement}

\definecolor{OliveGreen}{HTML}{447700}
\newcommand{\oldnotes}[1]{\opt{oldnotes: oldnotes}{#1}}

\def\gpp{g_{\text{\rm pp}}}

\newcommand{\coal}{\operatorname{Coal}}
\newcommand{\coalalt}{\widehat{\operatorname{Coal}}}
\newcommand{\coalpt}[1]{\mathscr{Z}(#1)}
\newcommand{\walk}[1]{\mathscr{X}_{#1}}
\newcommand{\future}[1]{\operatorname{FUT}(#1)} %

\def \path {\operatorname{PATH}}

\newcommand{\gexp}[1][]{g_{\text{Exp}}}

\newcommand{\directions}{\mathcal{U}}

\DeclareMathOperator{\supp}{supp}

\newcommand{\minw}{\overline{m}} %

\newcommand{\ed}{e_d} %
\newcommand{\ead}{e_{ad}} %

\newcommand{\origin}{0}

\usepackage[urlcolor=blue,colorlinks=true,linkcolor=blue,citecolor=blue]{hyperref}

\begin{document}
\title{First pair coalescence among three geodesics}
\author[A.~Krishnan]{Arjun Krishnan}
\address{Arjun Krishnan\\ University of Rochester\\  Department of Mathematics\\ Rochester, NY 14627\\ USA.}
\email{arjun@shirleyarjun.net}
\urladdr{https://shirleyarjun.net/homepage}
\thanks{A.\ Krishnan was partially supported by a Simons Collaboration grant 638966.}

\author[F.~Rassoul-Agha]{Firas Rassoul-Agha}
\address{Firas Rassoul-Agha\\ University of Utah\\  Mathematics Department\\ 155S 1400E\\   Salt Lake City, UT 84112\\ USA.}
\email{firas@math.utah.edu}
\urladdr{https://www.math.utah.edu/~firas}
\thanks{F.\ Rassoul-Agha was partially supported by National Science Foundation grants DMS-2054630 and DMS-2450951 and Simons Foundation grant MPS-TSM-00013661.}

\author[T.~Sepp\"al\"ainen]{Timo Sepp\"al\"ainen}
\address{Timo Sepp\"al\"ainen\\ University of Wisconsin-Madison\\  Mathematics Department\\ Van Vleck Hall\\ 480 Lincoln Dr.\\   Madison WI 53706-1388\\ USA.}
\email{seppalai@math.wisc.edu}
\urladdr{https://people.math.wisc.edu/~seppalai/}%
\thanks{T.~Sepp\"al\"ainen was partially supported by National Science Foundation grants  DMS-1854619, DMS-2152362, and DMS-2448375, by Simons Foundation grant 1019133, and by the Wisconsin Alumni Research Foundation.} 
\begin{abstract}
In directed last-passage percolation with i.i.d.\ exponential vertex weights, the coalescence time $\coal(a,b)$ of two semi-infinite geodesics started from vertices  $a$ and $b$ in a fixed direction has a heavy tail with (KPZ) exponent $2/3$, and hence infinite expectation. We study instead the first coalescence among triples of semi-infinite geodesics in first- and last-passage percolation models (FPP/LPP) on $\Z^2$ that arise from stationary compatible families, such as Busemann geodesics. For directed FPP and LPP with arbitrary weights, we prove that
\[\E\bigl[\min\{\coal(a,a-e_1),\coal(a,a-e_2)\}\bigr]<\infty.\]
We also consider FPP and LPP models in which antidiagonal motion is allowed, and prove that for every integer $k>0$,
\[\E\bigl[\min\{\coal(a,a-k(e_1-e_2)),\coal(a,a+k(e_1-e_2))\}\bigr]<\infty.\]
These results support the heuristic prediction that adjacent coalescence times are negatively associated.
\end{abstract}
\maketitle

\section{Introduction}
Directed first- and last-passage percolation (FPP/LPP) are models for random growth on a graph. In this paper we focus on graphs whose vertex set is \(\Z^2\). Attached to each vertex is a set $E$ of outgoing edges, which we take to be either $E=\{e_1,e_2\}$ or $E=\{e_1,e_2,\pm \ead\}$.  Here  $e_1, e_2$ are the standard basis vectors and $\ead = e_1 - e_2$ points in the  antidiagonal direction.  Later we also make use of the diagonal direction vector 
$\ed = e_1 + e_2$. A set of random weights $\{\w_{x,x+e}\}_{e \in E}$ is attached to the outgoing edges from each vertex $x$. Their  probability distribution is denoted by $\Prob$ and expected value is denoted by $\E$.

An \emph{admissible path} is a sequence of points in $\Z^2$ whose successive increments belong to $E$. We say that \(x \prec  y\) if  \(x\) and \(y\) can be connected by an admissible path \(\Gamma = \{x=x_0,x_1, \ldots, x_k=y\}\).
The \textbf{passage-time} of a finite admissible path \(\Gamma\) is 
\[W(\Gamma) = \sum_{i = 0}^{k-1} \w_{x_i,x_{i+1}},\]
and the \textbf{first-passage time} from $x$ to $y$ is
\begin{align}\label{eq:passage time definition}
    G(x,y) = \inf_{\Gamma} W(\Gamma),
\end{align}
where the infimum is over all admissible paths from $x$ to $y$. A \textbf{geodesic} is an admissible path between points \(x\) and \(y\) that achieves the infimum in \eqref{eq:passage time definition}. An \textbf{infinite geodesic} is an infinite path that is a geodesic between any two of its points. 

Write $\bar\w_x=\{\w_{x,x+e}:e\in E\}$ for the environment at vertex $x$. 
When $E=\{e_1,e_2,\pm\ead\}$, we assume
\begin{align}\label{assumption}
\begin{split}
&\{\bar\w_x:x\in\Z^2\}\text{ are independent and identically distributed},\\
&\E[(\w_{0,e_i})^-]<\infty,\quad i\in\{1,2\},\\
&\forall M<\infty,\ \Prob(\w_{0,e_1}\ge M,\ \w_{0,e_2}\ge M)>0,\\
&\exists\, c\in(0,\infty)\,:\,\Prob(0\le \w_{0,\ead}\le c)=\Prob(0\le \w_{0,-\ead}\le c)=1,\quad\text{and}\\
&\Prob(\w_{0,\ead}+\w_{0,-\ead}>0)>0.
\end{split}
\end{align}
When $E = \{e_1,e_2\}$, we do not assume anything about the weights, and Theorem \ref{thm:the coalescence time of three vertices has finite expectation} applies regardless.

The nonnegativity of the antidiagonal weights ensures that antidiagonal loops have nonnegative cost and hence can be removed without increasing the passage time, while the last assumption ensures that sufficiently long antidiagonal excursions incur a positive linear cost almost surely. The integrability assumption ensures that exceptionally negative $e_1$ and $e_2$ edge weights encountered along such excursions grow only sublinearly and therefore cannot offset this linear cost. Lemma \ref{lem:finite-p2p-geo} then shows that point-to-point first-passage times are finite and are attained by finite geodesic paths. The upper bound $c<\infty$ provides uniform control on the cost of paths that use a bounded number of antidiagonal edges, which will be used in the comparison arguments below. The joint unboundedness of the $e_1$ and $e_2$ edge weights and the independence across vertices are technical assumptions needed for the proof.

\begin{remark}
  The last-passage time is obtained by replacing the infimum in \eqref{eq:passage time definition} with a supremum. Consequently, all of our definitions and results for first-passage percolation have direct last-passage analogues obtained by replacing the weights $\w_{x,x+e}$ with $-\w_{x,x+e}$ and reversing the corresponding inequalities. For example, the fourth assumption in \eqref{assumption} becomes $\Prob(-c\le \w_{0,\pm\ead}\le 0)=1$.
\end{remark}

If $E = \{e_1,e_2\}$, let $\directions = \{ \xi=(\xi_1, \xi_2)\in\R^2 \colon \xi_1 + \xi_2 = 1, \xi_1, \xi_2 > 0 \}$, and if $E = \{e_1,e_2,\pm \ead\}$, let $\directions =  \{ \xi \colon \xi_1 + \xi_2 = 1 \}$ denote the set of relevant macroscopic directions. If $\w_{0,e}$, $e\in E$, are in $L^p$ for some $p>2$, then the subadditive ergodic theorem ensures the existence and finiteness of the so-called \emph{time-constant}
$$
    \gpp(\xi)=\lim_{n \to \infty} \frac{G(0,[n\xi])}{n} \quad\text{a.s.\ and in }L^1,
$$
for each $\xi \in \directions$, where $[y] \in \Z^2$ denotes the coordinatewise integer part of $y$. See %
(2.8) in \cite{MR4982325}. Outside of the case of LPP with $E=\{e_1,e_2\}$ and exponential or geometric vertex weights (i.e., $\w_{x,x+e_1}=\w_{x,x+e_2}$), the time-constant is poorly understood.

Suppose $\{x_n \in \Z^2\}$ is such that $x_n\cdot\ed \to \infty$ and $x_n/(x_n\cdot\ed) \to \xi \in \directions$. A Busemann function in direction $\xi$ is defined by the limit 
\begin{equation}
    B(a,b) = \lim_{n \to\infty} \left( G(a,x_n) - G(b,x_n) \right),\quad a,b\in\Z^2,
    \label{eq:busemann-fn-definition}
\end{equation}
if it exists. When $E=\{e_1,e_2\}$ and weights are on the vertices and are i.i.d.\ with $p>2$ moments, \citet{MR3535900} proved, under the unproven assumption that $\gpp$ is differentiable on $\directions$, that Busemann functions exist almost surely and are stationary under lattice translations. Similar results are expected to hold for edge weights and also when $E=\{e_1,e_2,\pm\ead\}$, since admissible paths do not backtrack in the $\ed$-direction and cannot cross without meeting. More generally, \citet{MR4982325} treat percolation and random polymer models that include edge and vertex weights and both choices of $E$ considered here. Its Theorem 4.5 constructs \emph{generalized Busemann functions}, obtained roughly by replacing the almost sure limits in \eqref{eq:busemann-fn-definition} with weak limits. This construction requires an augmentation of the probability space, while preserving invariance under lattice translations, and does not require regularity assumptions on the shape-function. 

Given a generalized Busemann function, let
\begin{equation}
    \arrow(x) = \operatorname*{argmax}_{y \in E} \left( B(x,x + y) - \w_{x,x+y} \right)
    \label{eq:arrow definition}
\end{equation}
be the arrow at $x$. Ties in \eqref{eq:arrow definition} may be broken by choosing some fixed shift-invariant rule. Given any $x \in \Z^2$, we form an admissible path $\walk{x}$ as follows:  $\walk{x}(0) = x$, and $\walk{x}(n) = \walk{x}(n-1) + \arrow(\walk{x}(n-1))$ for $n \geq 1$. It is shown in \cite{MR4982325} that these walks satisfy the following properties:
\begin{enumerate}[label={\rm(\alph*)}, ref={\rm\alph*}]
\item\label{geo}[Geodesics] For all $x\in \Z^2$,
\begin{align}\label{walkisgeo}
\walk{x}\text{ is an infinite geodesic out of $x$}.
\end{align}
\item\label{stat}[Stationarity] For all $x,z \in \Z^2$ and $k \in \Z_{\geq 0}$,
\begin{equation}
    \walk{z}(\w, k) = x + \walk{z-x}(T^x \w, k),
    \label{eq:stationarity}
\end{equation}
\item\label{comp}[Compatibility] For all $z\in\Z^2$ and $k\in\Z_{\ge0}$,
\begin{equation}
    \walk{z}(\w,k+1) = \walk{\walk{z}(\w,1)}(\w,k) .
    \label{eq:compatibility}
\end{equation}
\item\label{backindep}[Backward-independence] For every $a \in \Z^2$, let $L_a=\{a+(k,-k):k\in\Z\}$ be the antidiagonal line through $a$. Then 
\begin{equation}\label{assumption:busemann geodesics are independent of the weights behind it}
\text{$\{\walk{z} \colon (z - a) \cdot \ed \geq 0 \}$ are independent of the weights strictly below $L_a$.}
\end{equation}
\end{enumerate}

Our main Theorem \ref{thm:the coalescence time of three vertices has finite expectation} below applies to any family of walks $\{\walk{z}:z\in\Z^2\}$ that satisfies stationarity \eqref{eq:stationarity} and compatibility \eqref{eq:compatibility}.
Theorem  \ref{thm: at least one of pair of geodesics distance k apart expected to coalesce in finite time}, on the other hand, requires assumptions \eqref{assumption} on the weights and all four properties \eqref{walkisgeo}-\eqref{assumption:busemann geodesics are independent of the weights behind it} of the walks.

Let $\coalpt{a,b} \in \Z^2$ be the coalescence point of $\walk{a}$ and $\walk{b}$. If they do not coalesce, we set it to be $\infty$. We define the coalescence time of the two walks as
\begin{equation}
    \coal(a,b) := \max \left( (\coalpt{a,b} - a) \cdot \ed , (\coalpt{a,b} - b) \cdot \ed \right),
    \label{eq:coalescence time definition}
\end{equation}
and $\coal(a,b) = +\infty$ if the walks do not coalesce. 
When $E = \{e_1,e_2\}$, $\coal(a,b)$ equals the maximum number of $e_1$ and $e_2$ steps taken by each walk before coalescing. An alternate coalescence time we could consider is the maximal total number of steps each walk takes before coalescing:
\begin{equation}
    \coalalt(a,b) := \max \{ k \colon \walk{a}(k) = \coalpt{a,b} \OR \walk{b}(k) = \coalpt{a,b} \}.
    \label{eq:alternate coalescence time definition}
\end{equation}
When $E=\{e_1,e_2\}$, we have $\coal(a,b) = \coalalt(a,b)$, but they can be significantly different when $E = \{ e_1, e_2, \pm \ead\}$. We cannot deal directly with $\coalalt(a,b)$ using the methods of this paper, and leave that to future work. 

It is an open problem to describe the behavior of the coalescence time in general FPP and LPP models. There has been a lot of recent progress in the solvable models in LPP; i.e., with $E=\{e_1,e_2\}$ and i.i.d.\ exponential or geometric vertex weights. In a beautiful paper, \citet{pimentel_duality_2013} proved that the coalescence time $T_m = \coal((0,m),(m,0))$ grows like $m^{\zeta}$, where $\zeta \geq 3/2$. The matching upper bound of $\zeta \leq 3/2$ was subsequently proved by \citet{basu_coalescence_2017}. Further, \citet{pimentel_duality_2013} gives heuristics that indicate that the limiting coalescence time $T_m/m^{3/2}$ has heavy-tailed behavior with the following exponent: 
\begin{equation}
    \liminf_{m \to \infty} \Prob( T_m/m^{3/2} \geq r) \approx \frac{1}{r^{2/3}}. 
    \label{eq:pimentel-guess-for-coalescence-time-scale}
\end{equation}
A much stronger version of equation \eqref{eq:pimentel-guess-for-coalescence-time-scale} was proved in \cite{MR4002528}. Consider any two points $a,b$ on the antidiagonal line $L_0 = \{ (x,y) \in \Z^2 \colon x + y = 0 \}$ of the form $a = (-[k^{2/3}],[k^{2/3}])$ and $b = ([k^{2/3}],-[k^{2/3}])$. Then, \cite[Theorem 2]{MR4002528} states that there exist positive constants $C_1,C_2,R_0$ such that for all $k > 0$ and $R > R_0$:
\begin{equation}
    C_1 R^{-2/3} \leq \Prob \left( \coal(a,b) > Rk  \right) \leq C_2 R^{-2/3}.
    \label{eq:coalescence time lower and upper bounds}
\end{equation}
Without using integrable probability methods, Shen and Sepp\"al\"ainen \cite{MR4125790} proved the same bounds in \eqref{eq:coalescence time lower and upper bounds}. In particular, these results imply that $\E[ \coal(a,b) ] = +\infty$.

We consider instead the coalescence times $\coal(0,-e_2)$ and $\coal(0,-e_1)$. Heuristically, these ought to be negatively associated since if $\walk{0}$ and $\walk{-e_2}$ take a long time to coalesce, then it is probably because they both wandered off in different directions and are quite far apart. This indicates that $\walk{0}$ is more likely to travel towards $\walk{-e_1}$, and thus more likely to coalesce with it. Thus, we ought to expect that for any $r > 0$
\begin{align*}
    \Prob(  \coal(0,-e_2) \geq r , \coal(0,-e_1) \geq r )
    & \leq \Prob(  \coal(0,-e_2) \geq r ) \Prob(  \coal(0,-e_1) \geq r ) \\
    & \leq \left( \frac{1}{r^{2/3}} \right)^2
\end{align*}
and hence 
\[
    \E[ \min( \coal(0,-e_1), \coal(0,-e_2) ) ] < \infty.
\]
Our first main result proves this last statement using a refined ``in-expectation'' version of the Burton-Keane argument. We remark that the opposite corner $(e_2,\origin,e_1)$ is not of interest, because under $E=\{e_1,e_2\}$, $\min(\coal(0,e_1),\coal(0,e_2)) = 1$ automatically. 
\begin{theorem}
    Consider $E=\{e_1,e_2\}$. Let $\{\walk{a}\}_{a \in \Z^2}$ be a family of walks satisfying stationarity \eqref{eq:stationarity} and compatibility \eqref{eq:compatibility}. 
    Let \[{T}_{a} = \min \left( \coal(a, a-e_1), \coal(a,a - e_2) \right)\] be the minimal coalescence time of the two adjacent pairs of vertices $(a,a-e_1)$ and $(a,a-e_2)$. Then, 
    $\E[ {T}_{a} ] < \infty$.
    \label{thm:the coalescence time of three vertices has finite expectation}
\end{theorem}
\begin{remark}
    Theorem \ref{thm:the coalescence time of three vertices has finite expectation} applies to all families of stationary, compatible geodesics in directed FPP and LPP with edge set $E = \{e_1,e_2\}$. No assumptions on the weights $\w_{x,e}$ are needed, except to ensure the existence of stationary compatible families of geodesics. For example, i.i.d.\ weights and $\E[|\w_{x,e}|^p] < \infty$ for some $p > 2$ are sufficient. See Theorem 4.5 of \citep{MR4982325} for more general assumptions.
\end{remark}

The next theorem strengthens the preceding result under an additional backward-independence assumption. However, we are only able to prove it for a model in which additional antidiagonal movement is allowed. This helps in the construction of barrier events, which is a key component of the argument; see Figure \ref{fig:the box R} and Lemma \ref{lem:the edge configuration on Ba blocks geodesics}. 

\begin{theorem}\label{thm: at least one of pair of geodesics distance k apart expected to coalesce in finite time}
   Suppose $E=\{e_1,e_2,\pm\ead\}$ and assume that the assumptions in \eqref{assumption} hold. Suppose we are given a family of geodesics $\{\walk{a}\}_{a\in\Z^2}$ satisfying the stationarity, compatibility, and backward-independence properties \eqref{walkisgeo}-\eqref{assumption:busemann geodesics are independent of the weights behind it}. For each $k \in \Z_{> 0}$, let
\begin{equation*}
    T_{a,k} = \min( \coal(a,a-k \ead), \coal(a, a + k \ead)).
\end{equation*}
   Then, $\E[ T_{a,k} ] < \infty$. 
\end{theorem}

\begin{remark}
Theorem \ref{thm: at least one of pair of geodesics distance k apart expected to coalesce in finite time} continues to hold under the additional assumption that
$\Prob(c-1\le \w_{0,\pm\ead}\le c)=1$
for $c>1$. Formally, letting $c\to\infty$ suppresses the antidiagonal edges and recovers the case $E=\{e_1,e_2\}$. However, our upper bound \eqref{aux702} on $\E[T_{a,k}]$ diverges as $c\to\infty$. %
Consequently, this limiting argument does not yield $\E[T_{a,k}]<\infty$ for $E=\{e_1,e_2\}$.
\end{remark}

\begin{acknowledgement}
    AI did not contribute to either the formulation or proof of any of these results; they were first presented at an AMS Conference in 2018. We used an LLM to help us draw Figure \ref{fig:the box R}. 
\end{acknowledgement}

\section{Proofs} 
\subsection{Proof of Theorem \ref{thm:the coalescence time of three vertices has finite expectation}}
The setup of the Burton-Keane argument considers a box with side length $N$ in $\Z^2$. A triple of walks is simply three distinct walks emanating from three distinct points on the lattice. We say that a triple does not coalesce if no pair in the triple coalesces. The classical argument first assumes, for the sake of contradiction, that there exist $O(N^2)$ many triples of walks in the box that do not coalesce almost surely. As long as the outer two walks do not coalesce with the inner walk in a triple, they ``protect'' the inner walk from walks in other triples. In other words, protected inner walks in two distinct triples cannot coalesce. Thus, each protected inner walk must be associated with a unique point on the boundary of the box. However, there are only $O(N)$ many points on the boundary and $O(N^2)$ protected inner walks, and this shows that triples of walks where the outer walks do not coalesce with the inner walk almost surely do not occur. This is the classical Burton-Keane ``lack of space'' argument, which was used by Licea and Newman \citep{MR1387641} in their study of noncrossing geodesic families. 

We show that triples of walks cannot stay apart with no pair coalescing for \textit{too long}, by introducing a parameter representing the coalescence time into the Burton-Keane argument: for a triple that is $k$ steps away from the boundary, the protection from the outer walks only needs to last for $k$ steps to ensure that the inner walk is associated with a unique point on the boundary. Then, applying the Burton-Keane argument and taking expectation completes the proof. The rigorous version of this argument follows.

Let $L_k$ be the set of vertices of $\Z^2$ on the line parallel to $\ead$ that passes through the point $k \ed$. \arjun{Added `set of' here} Consider a box $R \subset \Z^2$ of vertices that has sides aligned with $\ed \AND \ead$, and $n$ vertices on each of its sides. Precisely,
\begin{equation}
    R = \{ i \ed + j \ead \colon i \in \{0,\ldots,n-1\},~ j \in \{0,\ldots,n-1\} \}.
    \label{eq:definition of rectangular box}
\end{equation}

Consider the two pairs of vertices $(a-e_1, a) \AND (a,a - e_2)$ for $a \in \Z^2$. Let 
\begin{equation*}
     E_{a,k} = \{ T_a > k \}
\end{equation*}
 be the event that the two pairs of walks from these three vertices do not coalesce for at least $k + 1$ steps. On this event, we must have
\[
    \arrow(a-e_1) = e_2 \AND \arrow(a - e_2) = e_1,
\]
since otherwise at least one pair of vertices involved in $E_{a,k}$ coalesces immediately and  $T_a = 1$. 

\begin{lemma}\label{lem:controlling the distance to coalescence}
    For distinct points $a$ and $b$, suppose $E_{a,k} \bigcap E_{b,r}$ occurs. Then on this event, either $\walk{a}$ takes at least $k$ steps or $\walk{b}$ takes at least $r$ steps before coalescing with the other. Hence,
    $\coal(a,b) \geq \min(r,k)$.
\end{lemma}

\begin{proof}
    For $i=1 \OR 2$, consider the \textit{open} set $M(a,a-e_i) \subset \R^2$ formed as the interior of the region bounded by the following three continuous curves in $\R^2$: the first two curves are formed by joining the vertices in $\walk{a} \AND \walk{a-e_i}$ with straight line-segments up to their coalescence point $\coalpt{a,a-e_i}$, and the third is the line-segment between $a \AND a - e_i$. If $\coalpt{a,a-e_i} = \infty$, then $M(a,a-e_i)$ is an infinite open region bounded by the two walks $\walk{a} \AND \walk{a-e_i}$ and the line between $a$ and $a-e_i$. Let $M(a) = \bigcup_{i=1,2} M(a,a - e_i)$ be the union of the two regions associated with the triple of points $(a,a-e_1,a-e_2)$. 

    \oldnotes{Consider any point $x$ that is not inside $M(a)$. Then for $\walk{a}$ and $\walk{x}$ to coalesce, $\walk{x}$ must coalesce with either $\walk{a-e_1} \OR \walk{a-e_2}$ first, or only coalesce with $\walk{x}$ after $\walk{x}$ has coalesced with $\walk{a-e_i} \OR \walk{a-e_2}$.  Therefore their coalescence time is satisfies $\coal{a,b} \geq k$.}

    If $b\not\in M(a)$ (even if it is on the boundary of $M(a)$), then $\walk{b} \AND \walk{a}$ cannot coalesce before $\walk{a}$ has coalesced with at least one of $\walk{a-e_i}$. This is because there is a unique forward walk starting at each point. Consequently, on the event $E_{a,k}$, $\walk{a}$ must take at least $k$ steps before coalescing with $\walk{b}$, and thus $\coal(a,b)\ge k$.

    Similarly, if $a\not\in M(b)$ then $\coal(a,b)\ge r\ge\min(r,k)$.

    Assume now that $b\in M(a)$ and $a\in M(b)$. The first condition implies $\{b \succ a-e_1 \OR b \succ a-e_2\}$ and the second implies $\{a\succ b-e_1 \OR a\succ b-e_2\}$. These two in turn imply that $b$ is in $\{a\pm e_1,a\pm e_2,a+e_i-e_{3-i}, i=1,2\}$. (Recall that $a\ne b$.) 
But $a-e_i$ and $a+e_i-e_{3-i}$ are not inside $M(a)$, while if $b=a+e_i$ then $a$ is not inside $M(b)$. So in fact, $b\in M(a) \AND a\in M(b)$ cannot happen simultaneously. Thus, on the event $E_{a,k} \cap E_{b,r}$ either $\walk{a}$ takes at least $k$ steps or $\walk{b}$ takes at least $r$ steps before coalescing. The proof of the lemma is complete. 
\end{proof}

\begin{proof}[Proof of \Thmref{thm:the coalescence time of three vertices has finite expectation}]
    Consider the sum
    \[
        I := \sum_{k = 0}^{n-1} \sum_{a \in R \cap L_k} 1_{E_{a,2(n-k)}}.
    \]
    Suppose $a \in L_p$ and $b \in L_q$ are distinct points, and suppose both $E_{a,2(n-p)}$ and $E_{b,2(n-q)}$ occur. Then by \Lemref{lem:controlling the distance to coalescence}, we must have that either $\walk{a}$ takes at least $2(n-p)$ steps or $\walk{b}$ takes at least $2(n-q)$ steps before the two walks can coalesce. But this means that $\coalpt{ a,b }$ must lie outside of $R$. Let $\partial R$ denote the set of sites $i e_d+j e_{ad}\in R$ such that either $i \in \{0,n-1\}$ or $j \in \{0,n-1\}$ in \eqref{eq:definition of rectangular box}.   So $\walk{a}$ and $\walk{b}$ intersect the boundary $\partial R$, but their intersections $\walk{a} \cap \partial R$ and $\walk{b} \cap \partial R$ must be disjoint.  Since there are at most $3n$ points on $\partial R$ the walks can exit from, we must have
    \[
         I \leq 3n.
    \]
    Taking expectation and using the translation covariance of $E_{a,k}$, we get
    \begin{equation*}
        \E[I] = n \sum_{k=0}^{n-1} \Prob( E_{\origin,2(n-k)} ) \leq 3 n.
    \end{equation*}
    Since $E_{\origin,2(n-k)}=\{T_{\origin} > 2(n-k)\}$, reindexing the sum gives
    \[
        \sum_{k=1}^{n} \Prob( T_{\origin} > 2 k ) \leq 3.
    \]
    Since this bound is independent of $n$, letting $n \to \infty$ gives $\E[T_{\origin}/2] \leq \sum_{k \geq 0} \Prob( T_\origin/2 > k) < \infty$.
\end{proof}

\subsection{Proof of Theorem \ref{thm: at least one of pair of geodesics distance k apart expected to coalesce in finite time}}

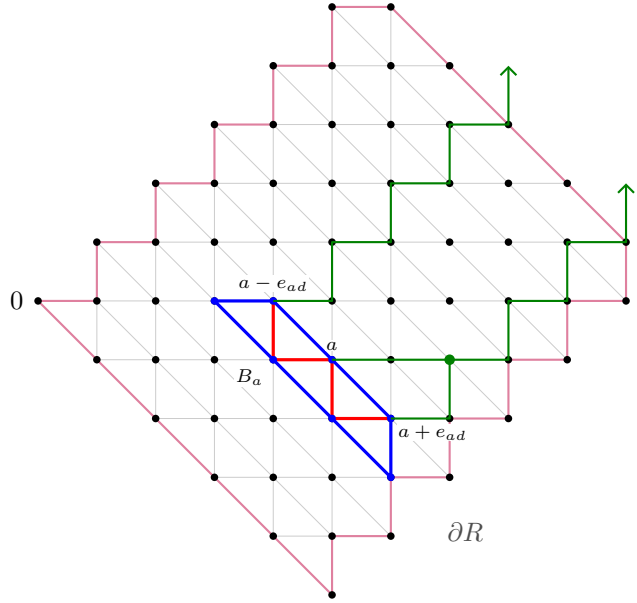
\begin{figure}[!hb]
    \begin{centering}
\ifdefined\figRn\else\gdef\figRn{6}\fi
\ifdefined\figRk\else\gdef\figRk{1}\fi
\ifdefined\figRai\else\gdef\figRai{2}\fi
\ifdefined\figRaj\else\gdef\figRaj{3}\fi
\ifdefined\figRscale\else\gdef\figRscale{1.1}\fi
\begin{tikzpicture}[scale=\figRscale, rotate=-45]

  \pgfmathtruncatemacro{\n}{\figRn}
  \pgfmathtruncatemacro{\k}{\figRk}
  \pgfmathtruncatemacro{\ai}{\figRai}
  \pgfmathtruncatemacro{\aj}{\figRaj}
  \pgfmathtruncatemacro{\m}{\n-1}   %
  \pgfmathtruncatemacro{\mm}{\n-2}  %

  \tikzset{
    lattice/.style = {gray!40, very thin},
    site/.style    = {circle, fill, inner sep=1.0pt},
    side/.style    = {blue, very thick, line cap=round},
    heavy/.style   = {red, very thick, line cap=round},
    boundary/.style= {purple!50, thick, line join=round},
    geod/.style    = {green!50!black, thick, line join=round,
                      -{Straight Barb[length=3pt]}},
  }

  \def\Rboundary{
    (0,0) -- (\m,0)
    \foreach \i in {1,...,\m} { -- (\m-0.5,\i-0.5) -- (\m,\i) }
    -- (\m-0.5,\m+0.5) -- (0.5,\m+0.5) -- (0,\m)
    \foreach \i in {\m,...,1} { -- (0.5,\i-0.5) -- (0,\i-1) }
  }

  \begin{scope}
    \clip \Rboundary -- cycle;
    \foreach \y in {0,0.5,...,\n} { \draw[lattice] (-1,\y) -- (\n,\y); }
    \foreach \c in {-\n,...,\n} {
      \draw[lattice] (\c,0) -- (\c+\n+1,\n+1);
      \draw[lattice] (\c,0) -- (\c-\n-1,\n+1);
    }
  \end{scope}

  \draw[boundary] \Rboundary -- cycle;

  \foreach \i in {0,...,\m} {
    \foreach \j in {0,...,\mm} { \node[site] at (\j+0.5,\i+0.5) {}; }
    \foreach \j in {0,...,\m}  { \node[site] at (\j,\i) {}; }
  }

  \pgfmathtruncatemacro{\lo}{\aj-\k+1}
  \pgfmathtruncatemacro{\hi}{\aj+\k}
  \foreach \x in {\lo,...,\hi} {
    \draw[heavy] (\x-0.5,\ai-0.5) -- (\x,\ai);    %
    \draw[heavy] (\x-0.5,\ai-0.5) -- (\x-1,\ai);  %
  }
  \draw[side] (\aj-\k,\ai) -- (\aj+\k,\ai)
              -- (\aj+\k+0.5,\ai-0.5) -- (\aj-\k-0.5,\ai-0.5) -- cycle;
  \foreach \x in {-\k,...,\k} { \node[site, blue] at (\aj+\x,\ai) {}; }
  \pgfmathtruncatemacro{\kp}{\k+1}
  \foreach \x in {-\k,...,\kp} { \node[site, blue] at (\aj+\x-0.5,\ai-0.5) {}; }

  \def\Rexit#1#2{\pgfmathtruncatemacro{\nsteps}{#1 > \m-1 ? 1 : 2*(\m-#2)+2}}

  \pgfmathtruncatemacro{\cl}{\aj-\k}
  \Rexit{\cl}{\ai}
  \draw[geod] (\aj-\k,\ai) \foreach \s [evaluate=\s as \dx using {ifthenelse(mod(\s,2)==1,0.5,-0.5)}] in {1,...,\nsteps} { -- ++(\dx,0.5) };
  \pgfmathtruncatemacro{\cr}{\aj+\k}
  \Rexit{\cr}{\ai}
  \draw[geod] (\aj+\k,\ai) \foreach \s [evaluate=\s as \dx using {ifthenelse(mod(\s,2)==1,0.5,-0.5)}] in {1,...,\nsteps} { -- ++(\dx,0.5) };

  \pgfmathtruncatemacro{\twok}{2*\k}
  \draw[geod, -] (\aj,\ai) \foreach \s in {1,...,\twok} { -- ++(0.5,0.5) };
  \node[circle, fill=green!50!black, inner sep=1.4pt] at (\aj+\k,\ai+\k) {};

  \ifnum\k=1 \def\kead{\ead}\else \def\kead{\k\ead}\fi
  \tikzset{lab/.style = {fill=white, inner sep=0.5pt, font=\scriptsize}}
  \node[lab, below left=3pt]  at (\aj-0.5,\ai-0.5) {$B_a$};
  \node[lab, above=3pt]       at (\aj,\ai)         {$a$};
  \node[lab, above=3pt]       at (\aj-\k,\ai)      {$a-\kead$};
  \node[lab, below right=2pt] at (\aj+\k,\ai)      {$a+\kead$};

  \node[left=2pt]            at (0,0)         {$\origin$};
  \node[black!70]            at (\m+0.6,\m/3) {$\partial R$};

\end{tikzpicture} 
    \end{centering}
    \caption{The box $R$ from the proof of Theorem \ref{thm: at least one of pair of geodesics distance k apart expected to coalesce in finite time} for $n = \figRn$, together with the box $B_a$ at $a = \figRai \ed + \figRaj \ead$ with $k = \figRk$. The black dots are the vertices of $R$ and the purplish outline is (the joined-up version of) $\partial R$. In $B_a$, the red internal weights are large, and taking detours along the blue exterior has lower weight than taking even one red internal edge. The green paths sketch the geodesics from $a - k \ead$, $a$ and $a + k \ead$; the ones from $a$ and $a + k \ead$ coalesce at the green dot.}
\label{fig:the box R}
\end{figure}
We define weights in the box of vertices 
\begin{equation}
B_a := \left\{ a + j \ead \colon j \in \{-k,\ldots,k\} \right\} 
\bigcup \left\{ a + j \ead - e_1 \colon j \in \{ -k,\ldots,k+1 \} \right\}.
\end{equation}
that will protect the geodesic $\walk{a}$ from behind; other geodesics must first coalesce with $\walk{a \pm k \ead}$ before coalescing with $\walk{a}$. Let $\minw\in\R$ be such that
\begin{equation*}
    \Prob(\w_{0,e_i}\le\minw)>0,\quad i\in\{1,2\}.
\end{equation*}
Define the event $F_{a,k}$ as follows: take $M>(2k+1)c+\minw$ and require that
\begin{enumerate}
\item On the blue surrounding box, the edges from $a - k \ead - e_1$ to $a - k \ead$ and $a + k \ead - e_2$ to $a + k \ead$ have weights $\le\minw$. 
\item Any $e_1$ or $e_2$ edge in the interior of the box has weight $\w_e \in [M,\infty)$.
\end{enumerate}
By our assumptions on the weights, $\Prob(F_{a,k}) > 0$.

\begin{lemma}\label{lem:the edge configuration on Ba blocks geodesics}
   On $F_{a,k}$, no geodesic can take an interior edge in $B_a$  \textup{(see Figure \ref{fig:the box R})}.
\end{lemma}

\begin{proof}
   Suppose a geodesic starts on the south boundary and takes an edge in the interior. Then it necessarily meets the north boundary of $B_a$ at some point $a + j \ead$, and the weight of this edge is at least $M$. Since $B_a$ has width $1$ in the $\ed$ direction, any alternate path that goes entirely along the edges of the box uses at most $2k+1$ antidiagonal edges, each of weight at most $c$, together with a single boundary $e_1$ or $e_2$ edge, whose weight is at most $\minw$. Such a path therefore has weight at most $(2k+1)c+\minw$. Since $M > (2 k + 1) c + \minw$, the alternative route around the boundary has lower weight, and so no such geodesic can exist. 
\end{proof}

\begin{proof}[Proof of \Thmref{thm: at least one of pair of geodesics distance k apart expected to coalesce in finite time}]
    To account for the presence of $\ead$ edges, let 
    \begin{multline}
        R = \{i \ed + j \ead \colon 0 \leq i \leq n-1, 0 \leq j \leq n-1\} \\
        \bigcup \{i \ed + j \ead + e_1 \colon 0 \leq i \leq n-2, 0 \leq j \leq n-2\}
    \end{multline}
    be a square box of vertices (see \Figref{fig:the box R}). Again, define $E_{a,j} = \{T_{a,k} > j\}$. Consider as before 
    \[
        I := \sum_{i = 0}^{n-1} \sum_{j = 1}^{\lfloor (n-1)/3k \rfloor} 1_{F_{i \ed + 3 k j \ead,k}} 1_{E_{i \ed + 3 k j \ead,2(n-i)}}.
    \]
   \begin{claim}\label{claim: the distance to coalescence}
        Suppose $a,b \in R$ are distinct points of the form $p \ed + 3k j\ead$ and $q \ed + 3k j'\ead$, and suppose $E_{a,2(n-p)}$, $F_{a,k}, E_{b,2(n-q)}$ and $F_{b,k}$ occur. Then on this event, either $\walk{a}$ takes at least $2(n-p)$ $e_1$ or $e_2$ steps or $\walk{b}$ takes at least $2(n-q)$ $e_1$ or $e_2$ steps before coalescing with the other. Hence,
        \[\coal(a,b) \geq \min(2(n-p),2(n-q)).\]
    \end{claim}
    
    \begin{proof}
        The proof is analogous to that of Lemma \ref{lem:controlling the distance to coalescence}. Define $M(a)$ as the open region formed by the $e_1$, $e_2$, and southwestern $\ead$-aligned parts of the blue boundary of the box $B_a$, and the walks $\walk{a-k\ead}$ and $\walk{a+k\ead}$ (with line segments connecting the vertices) all the way up to the coalescence point $\coalpt{a-k\ead,a+k\ead}$ (or taken to be the corresponding infinite open region if they do not coalesce). Define $M(b)$ similarly. See Figure \ref{fig:the box R}. 

       Without loss of generality, assume that $p \leq q < n$. Assume we are on the intersection of the events $E_{a,2(n-p)}$, $F_{a,k}, E_{b,2(n-q)}$ and $F_{b,k}$. If $p < q$, then $a \not\in M(b)$ since every point in $M(b)$ has $\ed$-coordinate strictly larger than $p$. If $p = q$, then $a$ could in the worst case lie on $\walk{b\pm k \ead}$, in which case it is on the boundary of $M(b)$; again $a \not\in M(b)$. The blocking event $F_{b,k}$ along with Lemma \ref{lem:the edge configuration on Ba blocks geodesics} shows that $\walk{a}$ cannot enter through the boundary of the box $B_b$ into its interior and must thus coalesce with either $\walk{b-k\ead}$ or $\walk{b+k\ead}$ before being able to coalesce with $\walk{b}$. The event $E_{b,2(n-q)}$ ensures that $\walk{b}$ takes at least $2(n-q)$ $e_1 \OR e_2$ steps before coalescing with either $\walk{b-k\ead}$ or $\walk{b+k\ead}$. This proves the claim.   
    \end{proof}
    
    Claim \ref{claim: the distance to coalescence} implies that $\coalpt{ a,b }$ must lie outside of $R$. As in the proof of Theorem \ref{thm:the coalescence time of three vertices has finite expectation}, each walk intersects $\partial R$, and the sets of boundary points visited by $\walk{a}$ and $\walk{b}$ are disjoint.  Therefore, choosing one such boundary point for each walk gives an injection into $\partial R$. Since there are at most $6n$ points on $\partial R$ through which the geodesics can exit, we must have $I \leq 6n$.

    Taking expectation and using the translation covariance of the events together with the independence of $E_{a,i}$ and $F_{a,k}$ by \eqref{assumption:busemann geodesics are independent of the weights behind it}, we get
    \begin{equation*}
        \E[I] = \lfloor (n-1)/3k \rfloor \Prob(F_{\origin,k}) \sum_{i=0}^{ n-1 } \Prob( E_{\origin,2(n-i) } ) \leq 6n.
    \end{equation*}
    Since $E_{\origin,2(n-i)}$ means $\{T_{\origin,k} > 2(n-i)\}$, we get for $n$ large enough,
    \begin{align}\label{aux702}
        \sum_{i=0}^{ n-1 } \Prob\left( T_{\origin,k} > 2 (n - i ) \right) \leq \frac{C}{\Prob(F_{\origin,k})}
    \end{align}
    for some constant $C$ that can be chosen to be close to $18k$. Reindexing by $m = n - i$ and letting $n \to \infty$ as in the proof of Theorem \ref{thm:the coalescence time of three vertices has finite expectation} shows
    \[\E[T_{\origin,k}/2] \leq \sum_{m=0}^\infty \Prob(T_{0,k}/2 > m) < \infty. \qedhere\]
\end{proof}

\appendix
\section{Finiteness of passage times and geodesics}
\begin{lemma}\label{lem:finite-p2p-geo}
Let $E=\{e_1,e_2,\pm\ead\}$ and $\bar\w_x=\{\w_{x,x+e}:e\in E\}$, $x\in\Z^2$. 
Suppose that $\{\bar\w_x:x\in\Z^2\}$ are 
i.i.d.\ and $\E[(\w_{0,e_i})^-]<\infty$, $i\in\{1,2\}$. Assume that
$\w_{x,x\pm\ead}\ge0$, $\Prob$-almost surely,
and $\Prob(\w_{0,\ead}+\w_{0,-\ead}>0)>0$.
 Then, $\Prob$-almost surely, simultaneously for
all $x,y\in\Z^2$ with $x\prec y$, the first-passage time $G(x,y)$ is finite and attained by a finite admissible path.
\end{lemma}

The intuition for the lemma was already provided in the introduction; we repeat it here for the reader's convenience. Since the antidiagonal edges have positive mean, the law of large numbers implies that the passage time along antidiagonal lines increases linearly in the number of steps. For any two fixed $x \prec y$, the total number of $e_1$ and $e_2$ edges on any allowable path is fixed and equal to $|x - y|_1$. Since $\E[ (\omega_{0,e_i})^-] < \infty$, the negative contributions from the $e_1$ and $e_2$ edges cannot compensate for the large positive linear growth of the antidiagonal lines passage time. Therefore, paths that have too many antidiagonal edges can be ignored, and only finitely many paths between $x$ and $y$ need to be considered when computing the first-passage time. This implies that the passage-time is finite and a geodesic exists. 

\begin{proof}
It is enough to prove the claim for an arbitrary pair $x\prec y$. Fix such a pair and set $N=(y-x)\cdot\ed$.
Every admissible path from $x$ to $y$ contains exactly $N$ edges in $\{e_1,e_2\}$.

Since the weights on antidiagonal edges are nonnegative, it suffices to consider self-avoiding paths.

For $s,r\in\Z$, write $v(s,r)=(r,s-r)\in\Z^2$.
Then $v(s,r)\cdot\ed=s$, and $v(s,r)+\ead=v(s,r+1)$.
Thus $r$ parametrizes the antidiagonal line of level $s$.

Define
\[A^+_{s,r}=\w_{v(s,r),v(s,r)+\ead}\wedge1
\quad\text{and}\quad
A^-_{s,r}=\w_{v(s,r),v(s,r)-\ead}\wedge1,\]
and set
\[\mu_+=\E[A^+_{0,0}]\quad\text{and}\quad\mu_-=\E[A^-_{0,0}].\]
By the assumptions on antidiagonal weights we have
$\mu_++\mu_->0$.

For each fixed $s$, the variables $\{A^\pm_{s,r}:r\in\Z\}$ are i.i.d.\ in $r$. Hence the strong law of large numbers gives
\[\max_{|r|\le R}
\Bigl|\sum_{j=0}^{r-1}(A^+_{s,j}-\mu_+)\Bigr|=o(R)\quad\text{and}\quad
\max_{|r|\le R}
\Bigl|\sum_{j=1}^{r}(A^-_{s,j}-\mu_-)\Bigr|=o(R),\]
$\Prob$-almost surely as
$R\to\infty$, with the usual interpretation of the sum for negative $r$.
Consequently, uniformly over all integers $u<v$ with
$|u|,|v|\le R$,
\begin{align}\label{aux748}
\sum_{r=u}^{v-1} A^+_{s,r}\ge \mu_+(v-u)-o(R)\quad\text{and}\quad
\sum_{r=u+1}^{v} A^-_{s,r}\ge \mu_-(v-u)-o(R).
\end{align}
Only the finitely many levels $s\in\{x\cdot\ed,\ldots,y\cdot\ed\}$
can be visited by an admissible path from $x$ to $y$, so these estimates hold uniformly over these $s$.

Given a self-avoiding admissible path $\gamma$, let $H_+(\gamma)$ and $H_-(\gamma)$ be the numbers of $+\ead$ and $-\ead$ steps in $\gamma$, respectively. Set $H(\gamma)=H_+(\gamma)+H_-(\gamma)$.

Let $\gamma$ be a self-avoiding admissible path from $x$ to $y$.
On each antidiagonal level, the antidiagonal portion of $\gamma$ is
monotone: if it changed direction, it would revisit a vertex. Hence, summing the bounds in \eqref{aux748} over the finitely many
levels visited by $\gamma$ gives 
\begin{align}\label{aux763}
W_{\rm ad}(\gamma)\ge\mu_+H_+(\gamma)+\mu_-H_-(\gamma) -o(H(\gamma)),
\end{align}
uniformly over all such $\gamma$, where $W_{\rm ad}(\gamma)$ denotes the contribution of the antidiagonal edge weights.

Let $N_1$ denote the number of $e_1$-steps in $\gamma$. Comparing first coordinates gives
$y_1-x_1=H_+(\gamma)-H_-(\gamma)+N_1$. Since $0\le N_1\le N$,
\[|H_+(\gamma)-H_-(\gamma)|\le |y_1-x_1|+N=:D.\]
Therefore
\[H_+(\gamma)\ge\frac{H(\gamma)-D}{2}\quad\text{and}\quad
H_-(\gamma)\ge\frac{H(\gamma)-D}{2},\]
and \eqref{aux763} yields
\begin{align}\label{aux775}
W_{\rm ad}(\gamma)\ge\frac{\mu_++\mu_-}{2}(H(\gamma)-D)-o(H(\gamma)).
\end{align}

We next control the possible negative contribution of the $e_1$ and $e_2$ edges. 
Since there are $O(2^m)$ candidate $e_i$ edges with antidiagonal coordinate $|r|\le 2^m$ on the finitely many relevant levels, a union bound gives
\[\Prob\biggl(
\max_{\substack{x\cdot\ed\le s\le y\cdot\ed\\ |r|\le 2^m}}(\w_{v(s,r),v(s,r)+e_i})^-
>\epsilon 2^m\biggr)
\le C_{x,y}\,2^m
\Prob\bigl((\w_{0,e_i})^->\epsilon2^m\bigr),\]
where $C_{x,y}$ is a finite positive constant.
By the integrability of $(\w_{0,e_i})^-$,
\[\sum_{m=1}^\infty
2^m\Prob\bigl((\w_{0,e_i})^->\epsilon 2^m\bigr)
\le \frac2\epsilon\E[(\w_{0,e_i})^-]<\infty.\]
Hence, by the Borel-Cantelli lemma, $\Prob$-almost surely,
\[\max_{\substack{x\cdot\ed\le s\le y\cdot\ed\\ |r|\le 2^m}}
(\w_{v(s,r),v(s,r)+e_i})^-
\le \epsilon 2^m\]
for all sufficiently large $m$. Since $\epsilon>0$ is arbitrary, it follows that
\begin{align}\label{aux798}
\max_{\substack{x\cdot\ed\le s\le y\cdot\ed\\ |r|\le R}}(\w_{v(s,r),v(s,r)+e_i})^-=o(R)
\qquad\Prob\text{-a.s.}
\end{align}

Every vertex of a self-avoiding admissible path $\gamma$ from $x$ to $y$ with $H$ antidiagonal steps has antidiagonal coordinate bounded in absolute value by
$H+C'_{x,y}$ for a deterministic constant $C'_{x,y}$. Since $\gamma$ has exactly $N$ ``forward'' edges $e_i$, $i\in\{1,2\}$, \eqref{aux798} gives 
\begin{align}\label{aux805}
W_{\rm f}(\gamma)\ge-o(H(\gamma)),
\end{align}
where $W_{\rm f}(\gamma)$ denotes the contribution of the $e_1$ and $e_2$ edges.

Combining \eqref{aux775} and \eqref{aux805},
\[W(\gamma)\ge
\frac{\mu_++\mu_-}{2}H(\gamma)-o(H(\gamma))-O_{x,y}(1).\]
Since $\mu_++\mu_->0$, we have almost surely
\begin{align}\label{aux817}
\inf\bigl\{
W(\gamma):
\gamma\text{ self-avoiding from $x$ to $y$},\ H(\gamma)\ge m
\bigr\}\longrightarrow\infty
\quad\text{as }m\to\infty.
\end{align}

Fix any finite admissible path $\gamma_0$ from $x$ to $y$. By \eqref{aux817},
there exists a finite random $m_0$ such that $W(\gamma)>W(\gamma_0)$ for all admissible paths $\gamma$ from $x$ to $y$ with $H(\gamma)\ge m_0$.
Hence the infimum defining $G(x,y)$ may be restricted to self-avoiding paths with fewer than $m_0$ antidiagonal edges. Every such path has at most $N+m_0$ edges, so there are only finitely many
of them. Consequently, $G(x,y)\in\R$ 
and the infimum is attained by one of these finitely many paths.
\end{proof}

\oldnotes{
\begin{claim}
    The event $E_{k,r}$ only depends on $\future{a - e_2} \cup \future{a - 2 e_2}$. In particular, it is independent of $\w_{a -2 e_2}$.
\end{claim} 

$A_{r,k}$ will be an event in a small wedge shaped box behind the vertices involved in $E_{r,k}$. The event $A_{k,r}$ is defined on Since $E_{r,k}$ is independent of the weights behind the three vertices involved in the event, it is independent of $A_{r,k}$. $A_{r,k}$ is defined on a square box of side-length $2$ that includes the vertices $V := \{ a - i e_1 - j e_2 \}_{i,j = 1,2} \setminus \{a, a - e_2\}$.

\begin{claim}
    The event $E_{k,r}$ only depends on $\future{a - e_2} \cup \future{a - 2 e_2}$. In particular, it is independent of $\w_{a -2 e_2}$.
\end{claim} 
}

\notes{
We hoped to prove that there is a generalization of \Thmref{thm:the coalescence time of three vertices has finite expectation} that actually relies on the i.i.d.\ structure of the weights to make a modification and lack of space argument in the spirit of Burton-Keane and Licea-Newman. We detail here why it doesn't work.

Actually in 2026 we have done it. But still we retain a part of it since it shows why it doesn't work in standard LPP.
\begin{lemma}
    $A_{k,x}$ is independent of $E_{k,x}$.
\end{lemma}
The lemma follows from the fact that $E_{k,x}$ only depends on Busemann increments that lie in $\future{x} \cup \future{x - k e_1} \cup \future{x - k e_2} $. Since $A_{k,x}$ depends on weights that lie in $F_{x,r}$, these two events are independent. \notes{This probably needs differentiability in direction $\xi$, associated with the Busemann function.}

\begin{theorem}[\emph{Not a theorem}]
    Suppose we have i.i.d.\, weights, and assume that the $B(x,y)$ is a measurable function of only the weights in $\future{x} \cup \future{y}$. Let $T_{k,a} = \min \left( \coal(a,a-k e_1), \coal(a,a - k e_2) \right)$ for $k > 1$. Then
    \[
        \E[T_{k,a}] < \infty.
    \]
    \label{thm:generalized finite expected coalescence time theorem with iid weights}
\end{theorem}

This appears to require a modification argument. We define an event $A_{k,x}$ on the weights in the rectangle $F_{x,r}$ defined by the four corners $x,x-re_1,x-r e_2,x-re_1 - r_2$. $r > k$ is a parameter (to be determined) that will depend on $k$ and the law of the weights. The weights will be chosen such that from any point on the south (west) boundary of $F_{x,r}$, Busemann geodesics may only travel up (right).  Precisely, let $A_{k,x}$ be the event that $\w_y \geq p$ for $p \in \supp(\w)$, for all $y$ on the south and west boundaries of $F_{x,r}$, and $\w_y \leq q$ for $q \in \supp(\w)$, such that $p$ and $q$ satisfy $(r-k)p > 2r q$. It's clear that for any fixed $k$, we may always find $p$, $q$ and $r$ as required above.

\begin{lemma}
    Let $\gamma$ be a Busemann geodesic that begins on the south or west boundary of $F_{x,r}$. Then, on the event $E_{k,x} \cap A_{k,x}$, cannot coalesce with $x$ before $\walk{x-ke_1}$ or $\walk{x-ke_2}$. 
    \label{lem:no geodesics passing through x}
\end{lemma}
The lemma implies the point $x$ is not connected to any point on the boundary of $F_{x,r}$ and hence is ``protected'' from other geodesics \notes{define more precisely}

\begin{proof}
    Suppose $\gamma$ begins on the west boundary of $x$, we show that $\gamma$ must either enter the north or south boundaries of $F_{x,r}$ at a point $a = x - l e_i$ such that $l \geq k$. This shows that $\gamma$ can only coalesce with $x$ after either $\walk{x - ke_i}$ coalesce with $x$.
    
    Suppose, $a = x - l e_1$. Then, $a$ must be to the right of $x - k e_1$, since otherwise it will have to pass through $\walk{x - ke_1}$, and must therefore only be able to coalesce with $\walk{x}$ after $\walk{x - k e_1}$ does. In this case, we compare the passage time of the $\gamma$ geodesic up to $a$ with the $\gamma'$ path that goes straight up to $x - r e_1$ and then goes horizontally to $a$. The $\gamma'$ path collects at least $(r - k)$ $p$ weights more that the $\gamma$ path. The $\gamma$ path could pick up at most $r$ weights that it does not share with $\gamma'$, and these must all be $q$ weights. Then, 
    \[
        G(\gamma') - G(\gamma) \geq (r - k) p - r q > 0.
    \]
    But since $\gamma$ was assumed to be a geodesic, this is a contradiction. 
    
    Therefore if the $\gamma$ starts on the west boundary at some point $a$, it must go through the east boundary of $F_{x,r}$ at some point $b$. But now, there is \emph{no contradiction to be had}, since we cannot construct an alternate path inside of $F_{x,r}$ that uses more $p$ vertices than $q$ vertices, and joins up with $\gamma$ at the point $b$! 
\end{proof}
}

\printbibliography

\end{document}